\documentclass[smallextended]{svjour3}     
\smartqed  
\usepackage{graphicx}
\usepackage{amssymb,latexsym}
\newcommand{\la}{\lambda}

\newcommand{\al}{\alpha}

\newcommand{\be}{\beta}

\newcommand{\refeq}[1]{(\ref{eqn:#1})}
\newcommand{\labt}[1]{\label{thm:#1}}
\newcommand{\reft}[1]{Theorem~\ref{thm:#1}}
\newcommand{\labl}[1]{\label{lemma:#1}}
\newcommand{\refl}[1]{Lemma~\ref{lemma:#1}}

\newcommand{\labd}[1]{\label{definition:#1}}
\newcommand{\refd}[1]{Definition~\ref{definition:#1}}
\newcommand{\labc}[1]{\label{coro:#1}}
\newcommand{\refc}[1]{Corollary~\ref{coro:#1}}

\newcommand{\e}{\epsilon}
\newcommand{\ob}{\frac {1} {b}}
\newcommand{\fsz}{{\partial f_i\over \partial z}}
\newcommand{\fsw}{{\partial f_i\over \partial w}}
\newcommand{\pbw}{P_{b,w}}
\newcommand{\gbb}{g_b(B)}
\newcommand{\ekn}{(\e,k,\nu_b)}

\journalname{Monatsh Math}
\begin{document}

\title{Cantor Series Constructions Contrasting Two Notions of Normality}


\author{C. Altomare        \and
        B. Mance
}


\institute{C. Altomare \at
              Department of Mathematics \\
              The Ohio State University \\
              231 West 18th Avenue \\
              Columbus, OH 43210-1174
              Tel.: (614) 292-1923\\
              Fax: (614) 292-1479\\
              \email{altomare@math.ohio-state.edu}           
           \and
B. Mance \at
              Department of Mathematics \\
              The Ohio State University \\
              231 West 18th Avenue \\
              Columbus, OH 43210-1174
              Tel.: (614) 292-6063\\
              Fax: (614) 292-1479\\
              \email{mance@math.ohio-state.edu}           
}

\date{Received: date / Accepted: date}

\maketitle

\begin{abstract}
A. R\'enyi \cite{Renyi} made a definition that gives a generalization of simple
normality in the context of $Q$-Cantor series.
In \cite{Mance}, a definition of $Q$-normality
was given that generalizes the
notion of normality in the context of $Q$-Cantor series.
In this work, we examine both $Q$-normality and $Q$-distribution normality, treated in \cite{Laffer} and \cite{Salat}.
Specifically, while the non-equivalence of these two notions is implicit in \cite{Laffer},
in this paper, we give an explicit construction witnessing the nontrivial
direction. That is, we construct a base $Q$ as well as a real $x$ that
is $Q$-normal yet not $Q$-distribution normal.
We next approach the topic of simultaneous normality by constructing an
explicit example of a base $Q$ as well as a real $x$ that is
both $Q$-normal and $Q$-distribution normal.

\end{abstract}

\section{Introduction}
\label{sec:1}

\begin{definition}\labd{1.1} Let $b$ and $k$ be positive integers.  A {\it block of length $k$ in base $b$} is an ordered $k$-tuple of integers in $\{0,1,\ldots,b-1\}$.  A {\it block of length $k$} is a block of length $k$ in some base $b$.  A {\it block} is a block of length $k$ in base $b$ for some integers $k$ and $b$.
\end{definition}

Given a block $B$, $|B|$ will represent the length of $B$. Given non-negative integers $l_1,l_2,\ldots,l_n$, at least one of which is positive, and blocks $B_1,B_2,\ldots,B_n$, the block
\begin{equation}
B=l_1B_1 l_2B_2 \ldots l_n B_n
\end{equation}
will be the block of length $l_1 |B_1|+\ldots+l_n |B_n|$ formed by concatenating $l_1$ copies of $B_1$, $l_2$ copies of $B_2$, through $l_n$ copies of $B_n$.  For example, if $B_1=(2,3,5)$ and $B_2=(0,8)$, then $2B_1 1B_20B_2=(2,3,5,2,3,5,0,8)$.

\begin{definition}\labd{1.2} Given an integer $b \geq 2$, the {\it $b$-ary expansion} of a real $x$ in $[0,1)$ is the (unique) expansion of the form
\begin{equation} \label{eqn:bary} 
x=\sum_{n=1}^{\infty} \frac {E_n} {b^n}=0.E_1 E_2 E_3 \ldots
\end{equation}
such that $E_n$ is in $\{0,1,\ldots,b-1\}$ for all $n$ with $E_n \neq b-1$ infinitely often.
\end{definition}

Denote by $N_n^b(B,x)$ the number of times a block $B$ occurs with its starting position no greater than $n$ in the $b$-ary expansion of $x$.

\begin{definition}\labd{1.3} A real number $x$ in $[0,1)$ is {\it normal in base $b$} if for all $k$ and blocks $B$ in base $b$ of length $k$, one has
\begin{equation} \label{eqn:bnormal1}
\lim_{n \rightarrow \infty} \frac {N_n^{b}(B,x)} {n}=b^{-k}.
\end{equation}
A number $x$ is {\it simply normal in base $b$} if \refeq{bnormal1} holds for $k=1$.
\end{definition}

 Borel introduced normal numbers in 1909 and proved that almost all (in the sense of Lebesgue measure) real numbers in $[0,1)$ are normal in all bases.  The best known example of a number that is normal in base $10$ is due to Champernowne \cite{Champernowne}.  The number
$$
H_{10}=0.1 \ 2 \ 3 \ 4 \ 5 \ 6 \ 7 \ 8 \ 9 \ 10  \ 11 \ 12 \ldots ,
$$
formed by concatenating the digits of every natural number written in increasing order in base $10$, is normal in base $10$.  Any $H_b$, formed similarly to $H_{10}$ but in base $b$, is known to be normal in base $b$. Since then, many examples have been given of numbers that are normal in at least one base.  One  can find a more thorough literature review in \cite{DT} and \cite{KuN}.

The $Q$-Cantor series expansion, first studied by Georg Cantor, is a natural generalization of the $b$-ary expansion.

\begin{definition}\labd{1.4} $Q=\{q_n\}_{n=1}^{\infty}$ is a {\it basic sequence} if each $q_n$ is an integer greater than or equal to $2$.
\end{definition}

\begin{definition}\labd{1.5} Given a basic sequence $Q$, the {\it $Q$-Cantor series expansion} of a real $x$ in $[0,1)$ is the (unique) expansion of the form
\begin{equation} \label{eqn:cseries}
x=\sum_{n=1}^{\infty} \frac {E_n} {q_1 q_2 \ldots q_n}
\end{equation}
such that $E_n$ is in $\{0,1,\ldots,q_n-1\}$ for all $n$
with $E_n \neq q_n-1$ infinitely often.
\end{definition}

Clearly, the $b$-ary expansion is a special case of \refeq{cseries} where $q_n=b$ for all $n$.  If one thinks of a $b$-ary expansion as representing an outcome of repeatedly rolling a fair $b$-sided die, then a $Q$-Cantor series expansion may be thought of as representing an outcome of rolling a fair $q_1$ sided die, followed by a fair $q_2$ sided die and so on.  For example, if $q_n=n+1$ for all $n$, then the $Q$-Cantor series expansion of $e-2$ is
$$
e-2=\frac{1} {2}+\frac{1} {2 \cdot 3}+\frac{1} {2 \cdot 3 \cdot 4}+\ldots
$$
If $q_n=10$ for all $n$, then the $Q$-Cantor series expansion for $1/4$ is
$$
\frac {1} {4}=\frac{2} {10}+\frac {5} {10^2}+\frac {0} {10^3}+\frac {0} {10^4}+\ldots
$$

For a given basic sequence $Q$, let $N_n^Q(B,x)$ denote the number of times a block $B$ occurs starting at a position no greater than $n$ in the $Q$-Cantor series expansion of $x$. Additionally, define
\begin{equation} \label{eqn:Qnk}
Q_n^{(k)}=\sum_{j=1}^n \frac {1} {q_j q_{j+1} \ldots q_{j+k-1}}.
\end{equation}

A. R\'enyi \cite{Renyi} defined a real number $x$ to be normal with respect to $Q$ if for all blocks $B$ of length $1$,
\begin{equation}\label{eqn:rnormal}
\lim_{n \rightarrow \infty} \frac {N_n^Q (B,x)} {Q_n^{(1)}}=1.
\end{equation}
If $q_n=b$ for all $n$, then \refeq{rnormal} is equivalent to {\it simple normality in base $b$}, but not equivalent to {\it normality in base $b$}.  Thus, we want to generalize normality in a way that is equivalent to normality in base $b$ when all $q_n=b$.

\begin{definition}\labd{1.7} A real number $x$ is {\it $Q$-normal of order $k$} if for all blocks $B$ of length $k$,
\begin{equation}
\lim_{n \rightarrow \infty} \frac {N_n^Q (B,x)} {Q_n^{(k)}}=1.
\end{equation}
We say that $x$ is {\it $Q$-normal} if it is $Q$-normal of order $k$ for all $k$.
\end{definition}

We make the following definitions:

\begin{definition}\labd{1.8} A basic sequence $Q$ is {\it $k$-divergent} if
\begin{equation}
\lim_{n \rightarrow \infty} Q_n^{(k)}=\infty.
\end{equation}
$Q$ is {\it fully divergent} if $Q$ is $k$-divergent for all $k$.
\end{definition}

\begin{definition}\labd{1.6} A basic sequence $Q$ is {\it infinite in limit} if $q_n \rightarrow \infty$.
\end{definition}

For $Q$ that are infinite in limit,
it has been shown that the set of all $x$ in $[0,1)$ that are $Q$-normal of order $k$ has full Lebesgue measure if and only if $Q$ is $k$-divergent \cite{Renyi}.  Therefore if $Q$ is infinite in limit, then the set of all $x$ in $[0,1)$ that are $Q$-normal has full Lebesgue measure if and only if $Q$ is fully divergent.  Additionally, given an arbitrary non-negative integer $a$, F. Schweiger \cite{Sch} proved that for almost every $x$ with $\epsilon > 0$, one has
$$
N_n( (a), x ) = Q_n^{(1)} + O\left( \sqrt{Q_n^{(1)}} \cdot  \log^{3/2+\epsilon} Q_n^{(1)}     \right).
$$

It is more difficult to construct specific examples of $Q$-normal numbers than it is to show that the typical real number is $Q$-normal.  This is similar to the case of the $b$-ary expansion.  The situation is more complicated when $Q$ is infinite in limit as we  need to consider blocks whose digits come from an infinite set.

For example, normality can be defined for the continued fraction expansion,  which involves an infinite digit set.  While it is known that almost every real number is normal with respect to the continued fraction expansion, there are not many known examples (see \cite{AKS} and \cite{PostPyat}).

Generally speaking, it is more difficult to give explicit constructions
of normal numbers (for various notions of normality) than it is to give
typicality results. In \cite{Mance}, the second author gave an explicit construction
of a basic sequence $Q$ and a real number $x$ such that $x$ is $Q$-normal.
For the same $Q$ and $x$, we show that
$x$ is $Q$-distribution normal, a term we now define. First, we must define
$T_{Q,n}(x)$.

\begin{definition}
Let~$x$ be a number in $[0,1)$ and let~$Q$ be a basic sequence,
then $T_{Q,n}(x)$ is defined as
$$q_1\cdots q_n x\pmod{1}.$$
\end{definition}

\begin{definition}
A number~$x$ in $[0,1)$ is {\it $Q$-distribution normal} if
the sequence $\{T_{Q,n}(x)\}_{n=0}^\infty$ is uniformly distributed in $[0,1)$.
\end{definition}

Note that in base~$b$, where $q_n=b$ for all $n$,
 the notions of $Q$-normality and
$Q$-distribution normality are equivalent. This equivalence
is the most basic and fundamental fact in the study of normality
in base $b$. It is surprising that this
equivalence breaks down in the more general context of $Q$-Cantor series
for general $Q$.

\section{Block Friendly Families}
\label{sec:2}

We will state a theorem that allows us to construct specific examples of $Q$-normal numbers for certain $Q$.  We first need several definitions:

\begin{definition}\labd{1.9}\footnote{\cite{Post} discusses normality in base $2$ with respect to different weightings.} A {\it weighting} $\mu$ is a collection of functions $\mu^{(1)},\mu^{(2)},\mu^{(3)},\ldots$ with $\sum_{j=0}^{\infty} \mu^{(1)}(j)=1$ such that for all $k$, $\mu^{(k)}:\{0,1,2,\ldots\}^k \rightarrow [0,1]$ and $\mu^{(k)}(b_1,b_2,\ldots,b_k)=\sum_{j=0}^{\infty} \mu^{(k+1)}(b_1,b_2,\ldots,b_k,j)$.
\end{definition}

\begin{definition}\labd{1.10} The {\it uniform weighting in base $b$} is the collection $\lambda_b$ of functions  $\lambda_b^{(1)},\lambda_b^{(2)},\lambda_b^{(3)},\ldots$ such that for all $k$ and blocks $B$ of length $k$ in base $b$
\begin{equation}
\lambda_b^{(k)}(B)=b^{-k}.
\end{equation}
\end{definition}

\begin{definition}\labd{1.11} Let $p$ and $b$ be positive integers such that $1 \leq p \leq b$.  A weighting $\mu$ is {\it $(p,b)$-uniform} if for all $k$ and blocks $B$ of length $k$ in base $p$, we have
\begin{equation}
\mu^{(k)}(B)=\lambda_b^{(k)}(B)=b^{-k}.
\end{equation}
\end{definition}

Given blocks $B$ and $y$, let $N(B,y)$ be the number of
occurrences of the block~$B$ in the block~$y$.

\begin{definition}\labd{1.12} Let $\epsilon$ be a real number such that $0 < \epsilon < 1$ and let $k$ be  a positive integer.
Assume that $\mu$ is a weighting.  A block of digits $y$ is {\it $(\epsilon,k,\mu)$-normal }\footnote{\refd{1.12} is a generalization of the concept of $(\epsilon,k)$-normality, originally due to Besicovitch \cite{Besicovitch}.} if for all blocks $B$ of length $m \leq k$, we have
\begin{equation}
\mu^{(m)}(B)|y|(1-\epsilon) \le N(B,y) \le \mu^{(m)}(B)|y|(1+\epsilon).
\end{equation}
\end{definition}

For the rest of the paper we use the following conventions freely and without comment.  Given sequences of non-negative integers $\{l_i\}_{i=1}^\infty$ and $\{b_i\}_{i=1}^\infty$ with each $b_i\ge 2$
and a sequence of blocks $\{x_i\}_{i=1}^\infty$, we set
\begin{equation}
L_i=|l_1 x_1 \ldots l_i x_i|=\sum_{j=1}^i  l_j |x_j|,
\end{equation}
\begin{equation}
q_n=b_i \textrm{\ for $L_{i-1} < n \leq L_i$},
\end{equation}
and
\begin{equation}
Q=\{q_n\}_{n=1}^{\infty}.
\end{equation}
Moreover, if $(E_1,E_2,\ldots)=l_1x_1 l_2 x_2 \ldots$, we set
\begin{equation}
x=\sum_{n=1}^{\infty} \frac {E_n} {q_1 q_2 \ldots q_n}.
\end{equation}

Given $\{q_n\}_{n=1}^{\infty}$ and $\{l_i\}_{i=1}^{\infty},$ it is always
assumed that $x$ and $Q$ are given by the formulas above. 
We make the following definition of a block friendly family (BFF):

\begin{definition}\labd{1.13} A {\it BFF} is a $6$-tuple $W=\{(l_i,b_i,p_i,\epsilon_i,k_i,\mu_i)\}_{i=1}^{\infty}$ with non-decreasing sequences of non-negative integers $\{l_i\}_{i=1}^{\infty}$, $\{b_i\}_{i=1}^{\infty}$, $\{p_i\}_{i=1}^{\infty}$ and $\{k_i\}_{i=1}^{\infty}$, for which $b_i \geq 2$, $b_i \rightarrow \infty$ and $p_i \rightarrow \infty$, such that $\{\mu_i\}_{i=1}^{\infty}$ is a sequence of $(p_i,b_i)$-uniform weightings and $\{\epsilon_i\}_{i=1}^{\infty}$ strictly decreases to $0$.
\end{definition}

\begin{definition}\labd{1.14} Let $W=\{(l_i,b_i,p_i,\epsilon_i,k_i,\mu_i)\}_{i=1}^{\infty}$ be a BFF.  If $\lim k_i=K<\infty$, then let $R(W)=\{0,1,2,\ldots,K\}$. Otherwise, let $R(W)=\{0,1,2,\ldots\}$.
A sequence $\{x_i\}_{i=1}^\infty$ of $(\e_i,k_i,\mu_i)$-normal blocks of
non-decreasing length
is said to be {\it $W$-good} if for all $k$ in $R$,
the following three conditions hold:
\begin{equation}\label{eqn:good1}
\frac {b_i^k} {\epsilon_{i-1}-\epsilon_i}=o(|x_i|);
\end{equation}
\begin{equation}\label{eqn:good2}
\frac {l_{i-1}} {l_i} \cdot \frac {|x_{i-1}|} {|x_i|}=o(i^{-1}b_i^{-k});
\end{equation}
\begin{equation}\label{eqn:good3}
\frac {1} {l_i} \cdot \frac {|x_{i+1}|} {|x_i|}=o(b_i^{-k}).
\end{equation}
\end{definition}

We now state a key theorem of \cite{Mance}.

\begin{theorem}\labt{oldmain} Let $W$ be a BFF and $\{x_i\}_{i=1}^{\infty}$ a $W$-good sequence. If~$k~\in~R(W)$, then $x$ is $Q$-normal of order $k$.  If $k_i \rightarrow \infty$, then $x$ is $Q$-normal.
\end{theorem}

\section{$Q$-Normality Without $Q$-Distribution Normality: A Construction}
\label{sec:3}

In this section, we construct a specific example of
a basic sequence~$Q$ and a real number~$x$ such that~$x$ is $Q$-normal
yet not $Q$-distribution normal. Moreover, the $Q$-distribution
normality of~$x$ fails in a particularly strong fashion. 
Not only does $\{T_{Q,n}(x)\}_{n=1}^\infty$
fail to be u.d. mod $1$, but $\lim_{n\to\infty}T_{Q,n}(x)=0$.

We will use the following theorem of \cite{Salat}:

\begin{theorem}\labt{Salat}
Given a basic sequence $Q$ and a real number $x$ with
$Q$-Cantor series expansion
$x=\sum_{n=1}^\infty {E_n\over q_1\cdots q_n}$; if
$$\lim_{N \to\infty}{1\over N}\sum_{n=1}^N {1\over q_n}=0,$$
then
$x$ is $Q$-distribution normal iff
$$\left\{ {E_n\over q_n}\right\}_{n=1}^\infty$$
is u.d. mod $1$.
\end{theorem}

It should first be noted that it is easier to construct a basic sequence $Q$ and a real number $x$ that is
$Q$-distribution normal, but not $Q$-normal.  To see this, we let
$$
(E_1,E_2,\ldots)=(1,1,2,1,2,3,1,2,3,4,\ldots) \hbox{ \ and}
$$
$$
(q_1,q_2,\ldots)=(2,3,3,4,4,4,5,5,5,5,\ldots).
$$
Thus, the number $x=\sum_{n=1}^{\infty} \frac {E_n} {q_1 \ldots q_n}$ is not $Q$-normal since none of
the digits $\{E_n\}$ are equal to $0$.  However, $x$ is $Q$-distribution normal by \reft{Salat} since the
sequence $1/2,1/3,2/3,1/4,2/4,3/4,1/5,\ldots$ is u.d. mod $1$.

The construction of a basic sequence $Q$ and a real number $x$ that is $Q$-normal but not $Q$-distribution normal is far more difficult. We will first need to define a sequence of weightings
$\nu_1,\nu_2,\ldots $ and blocks $P_{b,w}$. After this, we will prove
a number of technical lemmas from which the above stated facts follow.

If we let~$b$ be a positive integer, then we define

\begin{displaymath}
\nu_b^{(1)}(j)=\left\{ \begin{array}{ll}
\frac {1} {2^b} & \textrm{if $0 \leq j \leq b-1$}\\
\frac {2^b-b} {2^b} & \textrm{if $j=b$}\\
0		& \textrm{if $j>b$}
\end{array} \right. .
\end{displaymath}
For a block $B=(b_1,\ldots,b_k)$, we define
$$\nu_b^{(k)}(B)=\prod_{j=1}^k \nu_b^{(1)}(b_j).$$

Note that $\nu_b$ is a $(b,2^b)$-uniform weighting. Since each $\nu_b^{(k)}$ is determined
by $\nu_b^{(1)}$, we  refer to $\nu_b^{(k)}$ as $\nu_b$ throughout.

Next, we define $\pbw$. Let~$b$ and~$w$ be positive integers.
Denote by $P_1,P_2,\ldots,P_{(b+1)^w}$ the blocks in base $b+1$
of length~$w$ written in lexicographic order. Let
$$\pbw=
2^{bw}\nu_b(P_1)P_1 2^{bw}\nu_b(P_2)P_2\cdots 2^{bw}\nu_b(P_{(b+1)^w})P_{(b+1)^w}.$$

In order to get upper and lower bounds for $N(B,\pbw)$
for a base $b+1$ block~$B$,
we need to calculate the length of $\pbw$. 
We must first compute $\nu_b(B)$. This calculation is facilitated
by the following definition:

\begin{definition}
Given a base $b+1$ block $B=(b_1,\ldots,b_w)$,
set
$$\gbb=|\{j ; b_j=b\}|.$$
\end{definition}

\begin{lemma}\labl{amount}
If $B$ is a base $b+1$ block of length $w$, then
$$2^{bw}\nu_b(B)=\left(2^b-b \right)^{\gbb}.$$
\end{lemma}

\begin{proof}
$$2^{bw}\nu_b(B)=
2^{bw}  \cdot \left({2^b-b\over 2^b}\right)^{\gbb}
 \cdot \left({1\over 2^b}\right)^{w-\gbb} =
\left(2^b-b \right)^{\gbb}.$$\qed
\end{proof}

\begin{lemma}\labl{pbw}
If~$b$ and~$w$ are positive integers, then
$$|\pbw|=w\cdot 2^{bw}.$$
\end{lemma}

\begin{proof}
Fix~$m$ such that $0\le m\le w$. Clearly, the number of~$i$
such that $g_b(P_i)=m$ is ${w\choose m}b^{w-m}$. By \refl{amount} and the
definition of $\pbw$, each block $P_i$ is
concatenated $\left(2^b-b \right)^m$ times in forming $\pbw$, with each one of
these blocks having length~$w$.  It follows that the total number of digits contained
in all copies of each block~$P_i$ is
$$w\cdot {w\choose m}\cdot \left(2^b-b \right)^m\cdot b^{w-m}.$$

In order to obtain an expression for the length $|\pbw|$ of $\pbw$, we sum over all possible values of~$m$. Therefore
$$|\pbw|=\sum_{m=0}^w w\cdot{w\choose m} \cdot \left(2^b-b \right)^m\cdot b^{w-m}
=w\cdot(2^b-b+b)^w=w\cdot 2^{bw}$$
by the binomial theorem.\qed
\end{proof}

\begin{lemma}\labl{ng}
Let~$w$, $k$, and~$b$ be positive integers such that $k\le w$.
If~$B$ is a block of length~$k$ in base $b+1$, then
$$N(B,\pbw)\ge (w-k+1)\cdot \left(2^b-b \right)^{\gbb}\cdot 2^{b(w-k)}.$$
\end{lemma}

\begin{proof}
$\pbw$ is defined as the concatenation of copies of the
blocks~$P_i=(p_{i,1},\ldots,p_{i,w})$. In order to get
this lower bound on $N(B,\pbw)$ it is enough to show that
the number of occurrences of~$B$ inside some copy of some~$P_i$ is exactly
$$(w-k+1)\cdot \left(2^b-b \right)^{\gbb}\cdot 2^{b(w-k)}.$$

Consider a block $P_i$ containing~$B$. Since~$B$ starts at position~$s$ in~$P_i$ for some~$s$ such
that $1\le s\le w-k+1$, this leaves exactly $w-k$ digits of~$P_i$
undetermined. Let
$$M=\{j ; p_{i,j}=b\hbox{ and } j\not\in[s,s+k-1] \}$$
and let $m=|M|$.
Thus $m$ is the number of times that the block~$P_i$ takes
on the value~$b$ outside of~$B$. Clearly $0\le m\le w-k$.

Note that since $m$ is the number of $b$'s in~$P_i$ outside $B$ and
$\gbb$ is number of $b$'s in~$P_i$ inside $B$, we see that $\gbb+m$ is the
total number of $b$'s in~$P_i$. By \refl{amount}, 
exactly $\left(2^b-b \right)^{\gbb+m}$ copies of $P_i$ are concatenated in forming $\pbw$.
Let $S$ be the total number of occurrences of $B$ in blocks $P_i$ that have exactly $m$ occurrences of $b$ outside of $B$.
Since there are $w-k+1$ choices for $s$, ${w-k\choose m}$ choices
for~$M$, $w-k-m$ undetermined positions after choosing~$M$,
and each undetermined position has~$b$ possible values, we see that
$$S = (w-k+1) \cdot {w-k\choose m} \cdot \left(2^b-b \right)^{\gbb+m}\cdot b^{w-k-m}.$$
So, to count the number of times that~$B$ occurs in~$\pbw$,
we sum over $m$ from $0$ to $w-k$ and use the
binomial theorem to get
$$N(B,\pbw)\ge
\sum_{m=0}^{w-k}(w-k+1) \cdot {w-k\choose m} \cdot \left(2^b-b \right)^{\gbb+m}\cdot b^{w-k-m}$$
$$=(w-k+1) \cdot \left(2^b-b \right)^{\gbb}\sum_{m=0}^{w-k}{w-k\choose m} \cdot \left(2^b-b \right)^m\cdot b^{w-k-m}$$
$$=(w-k+1) \cdot \left(2^b-b \right)^{\gbb} \cdot (2^b)^{w-k}.$$\qed
\end{proof}

We will need the following definition in the proof of \refl{nl}.

\begin{definition}
Let $B$, $C$, and $D$ be blocks with $|B|\ge 2$. Suppose that $B=(b_1,\ldots,
b_k)$, $C=(c_1,\ldots,c_m)$, and $D=(d_1,\ldots,d_t)$. We say that $B$
straddles $C$ and $D$ if there is an integer $s$ in $[2,k]$, an
integer $e$ in $[1,m]$, and an integer $f$ in [1,t] such that
$(b_1,\ldots, b_{s-1})=(c_e,\ldots,c_m)$ and
$(b_s,\ldots, b_k)=(d_1,\ldots,d_f)$.
\end{definition}

Intuitively, $B$ straddles $C$ and $D$ if $B$ starts in $C$ and ends in $D$.
It is worth noting that with this definition, if $|B|=1$ then there are no
choices of $C$ and $D$ for which $B$ straddles $C$ and $D$.

\begin{lemma}\labl{nl}
Let~$w$, $k$, and~$b$ be positive integers such that $k\le w$.
If~$B$ is a block of length~$k$ in base $b+1$, then
$$N(B,\pbw)
\le w\cdot \left(2^b-b \right)^{\gbb}\cdot 2^{b(w-k)}
+(k-1)(b+1)^w.$$
\end{lemma}

\begin{proof}
Note that $\pbw$ has the form $1C_11C_2\cdots 1C_t$ for some
length~$w$ blocks $C_1,\ldots, C_t$ and some~$t$.
In proving \refl{ng}, we showed that the number of
occurrences of~$B$ in some~$P_i$ is exactly
\begin{equation}\label{eqn:ledit1}
(w-k+1)\cdot \left(2^b-b \right)^{\gbb}\cdot 2^{b(w-k)}.
\end{equation}
When \refeq{ledit1} is added to an upper bound for the number of occurrences
of~$B$ in~$\pbw$ that straddle~$C_i$ and
$C_{i+1}$ for some $i$, we  obtain an upper bound for $N(B,\pbw)$.

Consider a block~$B$ that straddles the block~$C_i=(c_{i,1},\ldots,c_{i,w})$
and~$C_{i+1}=(c_{i+1,1},\ldots,c_{i+1,w})$ for some~$i$.
In this case,~$B$ starts at position~$s$ in~$C_i$ for some~$s$ such
that $w-k+2\le s\le w$.
Define $B_1=(c_{i,s},\ldots,c_{i,w})$,
$B_2=(c_{i+1,1},\ldots,c_{i+1,k-w+s-1})$,
$B_2'=(c_{i,1},\ldots,c_{i,k-w+s-1})$, and
$B_1'=(c_{i+1,s},\ldots,c_{i+1,w})$.
Note that since $k\le w$, these four sets are pairwise disjoint.

If $C_i=C_{i+1}$, then $B_1=B_1'$ and $B_2=B_2'$. Since the
blocks~$B_1$ and~$B_2'$ are both contained in~$C_i$ and
$$|B_1|+|B_2'|=|B_1|+|B_2|=|B|=k,$$
we see that~$k$ positions of~$C_i$ are determined. Thus there are
$w-k$ undetermined positions in $C_i$.

Let
$$M=\{j ; c_{i,j}=b\hbox{ and } j\not\in[k-w+s,s-1] \}$$
and let $m=|M|$.
Therefore $m$ is the number of times the block~$C_i$ takes
on the value~$b$ outside $B_2'\cup B_1$. We again note that $0\le m\le w-k$.

Since $m$ is the number of $b$'s in~$C_i$
not determined by $B$, we know that
$g_b(B_1)$ is number of $b$'s in~$C_i$ inside $B_1$ and
$g_b(B_2')$ is number of $b$'s in~$C_i$ inside $B_2'$. Thus
$$g_b(C_i)=g_b(B_1)+g_b(B_2')+m=g_b(B_1)+g_b(B_2)+m=\gbb+m$$
is the total of number of $b$'s in~$C_i$. By \refl{amount}, it follows that
exactly $\left(2^b-b \right)^{\gbb+m}$ copies of $C_i$ are concatenated in forming $\pbw$.
For a fixed~$m$, define~$S_m$ to be
the total number of occurrences of~$B$ straddling some
$C_i$ and $C_{i+1}$ such that $C_i=C_{i+1}$
that have exactly~$m$ occurrences of~$b$ not determined by~$B$.
Since there are $k-1$ choices for $s$, ${w-k\choose m}$ choices
for~$M$, $w-k-m$ undetermined positions after choosing~$M$
and each undetermined position has~$b$ possible values,
we see that for a fixed~$m$
$$S_m\le(k-1) \cdot {w-k\choose m} \cdot \left(2^b-b \right)^{\gbb+m}\cdot b^{w-k-m}.$$

To obtain an upper bound for the number of times~$B$ occurs in~$\pbw$
straddling some $C_i$ and $C_{i+1}$ such that $C_i=C_{i+1}$,
we need only sum over $m$ from $0$ to $w-k$ and use the
binomial theorem to get
$$S:=\sum_{m=0}^{w-k} S_m\le
(k-1)\sum_{m=0}^{w-k}{w-k\choose m} \cdot \left(2^b-b \right)^{\gbb+m}\cdot b^{w-k-m}$$
$$=\left(2^b-b \right)^{\gbb}(k-1)\sum_{m=0}^{w-k}{w-k\choose m}\left(2^b-b \right)^m b^{w-k-m}=(k-1)\left(2^b-b \right)^{\gbb}(2^b)^{w-k}.$$

Next, we let~$S'$ be the number of occurrences of~$B$ straddling
the blocks~$C_i$ and~$C_{i+1}$ such that $C_i$ and~$C_{i+1}$ are
not equal. 
Let $Z$ denote the set of all $i$ such that $C_i \neq C_{i+1}$.
Since the~$C_i$'s are written in lexicographic order, it
follows that~$Z$  has
no more elements than the number of base $b+1$ blocks of length~$w$.
So $Z$ has at most $(b+1)^w$ elements. For each $i$ in~$Z$,
there are at most~$k-1$ occurrences of~$B$ straddling $C_i$ and
$C_{i+1}$. Therefore
$$S'\le (k-1) \cdot (b+1)^w.$$

For each occurrence of~$B$ in~$\pbw$, either $B$ occurs inside~$C_i$ for some~$i$,
$B$ straddles some~$C_i$ and~$C_{i+1}$ for which $C_i=C_{i+1}$,
or $B$ straddles some~$C_i$ and~$C_{i+1}$ for which $C_i\not=C_{i+1}$.
We determined an upper bound for
the number of occurrences of $B$ inside some $C_i$ in \refl{ng}.
In the proof of the current lemma,
we showed that $S$ is an upper bound for the number of
occurrences of $B$ straddling some~$C_i$ and~$C_{i+1}$ for which $C_i=C_{i+1}$.
Also in the proof of the current lemma, we have seen that
$S'$ is an upper bound for the number of
occurrences of $B$ straddling some~$C_i$ and~$C_{i+1}$
for which $C_i\not=C_{i+1}$. Putting these three facts together, we see that

$$N(B,\pbw)\le
(w-k+1)\cdot \left(2^b-b \right)^{\gbb}\cdot 2^{b(w-k)}+S+S'$$
$$
\leq
(w-k+1)\cdot \left(2^b-b \right)^{\gbb}\cdot 2^{b(w-k)}
+(k-1) \cdot \left(2^b-b \right)^{\gbb}\cdot 2^{b(w-k)}+(k-1) \cdot (b+1)^w$$
$$
=w\cdot \left(2^b-b \right)^{\gbb}\cdot 2^{b(w-k)}+(k-1) \cdot (b+1)^w.$$\qed
\end{proof}

We now want to show that $\pbw$ is $\ekn$-normal. First
we need a technical lemma:

\begin{lemma}\labl{1021}
If $m$, $b$, $k$ and $w$ are positive integers such that $b\ge 6$ and
$m\le k \le w/2$, then $(m-1)(b+1)^w\le k\cdot 2^{b(w-m)}.$
\end{lemma}

\begin{proof}
Since $m\le k$ and $k\le w/2$, it follows that
\begin{equation}\label{eqn:blank}
1\ge 2^{b(-w/2+m)}
=
2^{-bw/2}\cdot 2^{mb}
=
(2^{-b/2})^w\cdot 2^{mb}
\end{equation}
\begin{equation}\label{eqn:nlank}
\ge
\left((b+1)2^{-b}\right)^w \cdot 2^{mb}
\ge
\left({m-1\over k}\right) \cdot {(b+1)^w 2^{bm}\over 2^{bw}},
\end{equation}
where \refeq{blank} to \refeq{nlank} is due to $b+1\le 2^{b/2}$ for $b\ge 6$.
Therefore
\begin{equation}\label{eqn:blanknlank}
1\ge\left({m-1\over k}\right) \cdot {(b+1)^w 2^{bm}\over 2^{bw}}.
\end{equation}
Multiplying both sides of \refeq{blanknlank} by
$k\cdot 2^{b(w-m)}$, the lemma follows.\qed
\end{proof}

\begin{lemma}\labl{eknu}
Let $b$, $k$ and $w$ be positive integers such that $b\ge 6$
and $k\le w/2$. If $\e={k\over w}$, then $\pbw$ is
$\ekn$-normal.
\end{lemma}

\begin{proof}
By definition, $\pbw$ is $\ekn$-normal if for
all blocks~$B$ in base $b+1$ of length~$m\le k$
\begin{equation}
\nu(B)|\pbw|(1-\epsilon) \le N(B,\pbw) \le \nu(B)|\pbw|(1+\epsilon).
\end{equation}
Therefore by \refl{ng} and \refl{nl}, it is enough to show that
\begin{equation}\label{eqn:six}
\nu(B)|\pbw|(1-\epsilon)\le (w-m+1)\cdot \left(2^b-b \right)^{\gbb}\cdot 2^{b(w-m)}
\end{equation}
and
\begin{equation}\label{eqn:seven}
w\cdot \left(2^b-b \right)^{\gbb}\cdot 2^{b(w-m)}
+(m-1)(b+1)^w\le\nu(B)|\pbw|(1+\epsilon).
\end{equation}
To show \refeq{six}, we write
$$
(1-\e)|\pbw|\nu_b(B)=
\left(1-{k\over w} \right)w\cdot 2^{bw}\left(2^b-b \right)^{\gbb}2^{-bm}
$$
$$
=(w-k) \cdot \left(2^b-b \right)^{\gbb} \cdot 2^{b(w-m)}<(w-m+1) \cdot \left(2^b-b \right)^{\gbb} \cdot 2^{b(w-m)}.
$$
Next, to show \refeq{seven}, we write
$$
w \cdot \left(2^b-b \right)^{\gbb} \cdot 2^{b(w-m)}+(m-1)(b+1)^w \leq 
w \cdot \left(2^b-b \right)^{\gbb} \cdot 2^{b(w-m)}+k \cdot 2^{b(w-m)}
$$
$$
\le w \cdot \left(2^b-b \right)^{\gbb} \cdot 2^{b(w-m)}+k \cdot \left(2^b-b \right)^{\gbb} \cdot 2^{b(w-m)}
=(1+\e)w \cdot \left(2^b-b \right)^{\gbb} \cdot 2^{b(w-m)},
$$
where the first inequality follows from \refl{1021}. \qed
\end{proof}

\begin{theorem}\labt{qnex}
For $i\le 5$, let $x_i=(0,1)$, $b_i=2$ and $l_i=0$.
If for $i\ge 6$ we let $x_i=P_{i,i^2}$, $b_i=2^i$ and
$l_i=2^{4i^2}$, then~$x$ is $Q$-normal.
\end{theorem}

\begin{proof}
For each $i\ge 1$,
we shall define numbers $p_i$, $k_i$, $\e_i$, and 
weightings $\mu_i$ in order to define a $BFF$ $W$ such
that~$\{x_i\}_{i=1}^\infty$ is $W$-good. Thus, we have only to
verify \refeq{good1}, \refeq{good2} and \refeq{good3} of \reft{oldmain} to conclude that~$x$ is
$Q$-normal.

For $i\le 5$, we define
$p_i=2$, $k_i=1$ and $\mu_i=\lambda_2$.
For $i\ge 6$, set $p_i=i$, $k_i=i$
and $\mu_i=\nu_i$.
Define $\e_1=.9$, $\e_1=.8$, $\e_1=.7$, $\e_1=.6$, $\e_1=.5$
and $\e_i=1/i$ for $i\ge 6$.
Let $W=\{(l_i,b_i,p_i,\epsilon_i,k_i,\mu_i)\}_{i=1}^{\infty}.$
We note that since $\mu_i$ is $(p_i,b_i)$-uniform, it follows
by definition that~$W$ is a $BFF$.

Since $\lim_{i\to\infty}k_i=\lim_{i\to\infty}i=\infty$,
we see that $R(W)$ is the set of all non-negative integers.
So, it is enough to show that conditions \refeq{good1}, \refeq{good2} and \refeq{good3} hold for all
non-negative integers~$k$. First note that
$|x_i|=i^2 \cdot 2^{i^3}$ for $i\ge 6$.

To show \refeq{good1}, note that
$$
\lim_{i \rightarrow \infty} |x_i| \Bigg/ \left( \frac {2^{ik}} {\frac {1} {i-1}-\frac {1} {i}} \right)=\lim_{i \rightarrow \infty} \frac {i^2 \cdot 2^{i^3}} {2^{ik} \cdot i(i-1)}=\infty.
$$

To show \refeq{good2}, notice that
$$\lim_{i \rightarrow \infty} \frac {l_{i-1}} {l_i} \cdot \frac {|x_{i-1}|} {|x_i|} \cdot i \cdot 2^{ik}
\le \lim_{i \rightarrow \infty} 1 \cdot \left( \frac {i-1} {i} \right)^2 \cdot \frac {2^{(i-1)^3+k i}} {2^{i^3}} \cdot i
\le \lim_{i \rightarrow \infty} 1 \cdot 2^{-3i^2+(3+k)i-1} \cdot i=0.
$$

And finally, to show \refeq{good3}, we write
$$
\lim_{i \rightarrow \infty} \frac {1} {l_i} \cdot \frac {|x_{i+1}|} {|x_i|} \cdot 2^{ik}
=\lim_{i \rightarrow \infty} \left( \frac {i+1} {i} \right)^2 \cdot \frac {2^{(i+1)^3+ki}} {2^{4i^2} \cdot 2^{i^3}}
\leq \lim_{i \rightarrow \infty} 2 \cdot 2^{-i^2+(3+k)i+1}=0.
$$

This shows that $\{x_i\}_{i=1}^\infty$ is $W$-good. Therefore $x$ is $Q$-normal by \reft{oldmain}.\qed
\end{proof}

\begin{theorem}\labt{t0}
If $\{x_i\}$, $\{b_i\}$, and $\{l_i\}$ are defined as in \reft{qnex}, then~$\lim_{n\to\infty} T_{Q,n}(x)=0$.
\end{theorem}

\begin{proof}
To prove \reft{t0} we use the trick which is usually used to prove the irrationality of $x$.  For more information see e.g. \cite{HT}. Note that
$$T_{Q,n}(x)=q_1\cdots q_n x \pmod 1
={E_{n+1}\over q_{n+1}}+{E_{n+2}\over q_{n+1}q_{n+2}}+\cdots$$
Given~$n$, define $j=j(n)$ as the unique integer satisfying
$$L_{j-1} < n+1 \le L_j.$$
Note that $q_{n+1}=b_j=2^j$ and $E_{n+1}\le j$ by construction.
Additionally, note that
$${E_{n+2}\over q_{n+1}q_{n+2}}+{E_{n+3}\over q_{n+1}q_{n+2}q_{n+3}}+\cdots
\le {1\over q_{n+1}}
\left[
{E_{n+2}\over q_{n+2}}+{E_{n+3}\over q_{n+2}q_{n+3}}+\cdots
\right]
\le {1\over q_{n+1}}\cdot 1
={1\over q_{n+1}}.$$
Therefore since $0 \le E_{n+1} \le j$, we see that
$$T_{Q,n}(x)
={E_{n+1}\over q_{n+1}}+
\left[{E_{n+2}\over q_{n+1}q_{n+2}}+{E_{n+3}\over q_{n+1}q_{n+2}q_{n+3}}+\cdots
\right]
\le
{j \over 2^j}+{1\over 2^j}\to 0.$$\qed
\end{proof}

\begin{corollary}\labc{notdn}
If $\{x_i\}$, $\{b_i\}$, and $\{l_i\}$ are defined as in \reft{qnex}, then
$x$ is not $Q$-distribution normal.
\end{corollary}

\section{A Construction Giving Simultaneous $Q$-Normality and $Q$-Distribution Normality}
\label{sec:4}

\begin{definition}
We say that $V=\{(l_i,b_i,\e_i)\}_{i=1}^\infty$ is a {\it modular friendly family ($MFF$)} if 
$\{l_i\}_{i=1}^\infty$ and $\{b_i\}_{i=1}^\infty$ are
non-decreasing sequences of non-negative integers with $b_i\ge 2$
such that $\{\e_i\}_{i=1}^\infty$ is a decreasing sequence of real
numbers in $(0,1)$ with $\lim_{i\to\infty}\e_i=0$.
\end{definition}

\begin{definition}
Let $V=\{(l_i,b_i,\e_i)\}_{i=1}^\infty$ be an $MFF$.
A sequence $\{x_i\}_{i=1}^\infty$ of $(\e_i,1,\lambda_{b_i})$-normal blocks of non-decreasing length
with $\lim_{i\to\infty}|x_i|=\infty$ is
said to be $V$-{\it nice} if the following two conditions hold:
\begin{equation}\label{eqn:nice1}
{l_{i-1}\over l_i}\cdot {|x_{i-1}|\over |x_i|}=o(1/i);
\end{equation}
\begin{equation}\label{eqn:nice2}
{1\over l_i}\cdot {|x_{i+1}|\over |x_i|}=o(1).
\end{equation}
\end{definition}

Throughout this section, we fix an $MFF$ $V=\{(l_i,b_i,\e_i)\}$
and a $V$-nice sequence of blocks $\{x_i\}$.
Moreover,
if $x_i=(x_{i,1},x_{i,2},\ldots,x_{i,|x_i|})$, then $y_i$ will be understood to stand for the sequence
$$
\left\{ \frac {x_{i,j}} {b_i} \right\}_{j=1}^{|x_i|}.
$$

Given finite sequences $y_1,\ldots,y_t$ and non-negative integers $l_1,\ldots,l_t$, the notation $l_iy_i$ denotes the concatenation of $l_i$ copies of $y_i$ and the notation $l_1y_1\ldots l_ty_t$ denotes the concatenation of the sequences $l_1y_1,\ldots,l_ty_t$.

Throughout the rest of the paper,
for a given $n$, the letter $i=i(n)$ is the unique integer satisfying

\begin{equation}\label{eqn:idef}
L_i < n \leq L_{i+1}.
\end{equation}

Given a sequence $z=(z_1,\ldots,z_n)$ in $[0,1)$ and 
$0<\gamma\le 1$, we define $A([0,\gamma),z)$ as
$$|\{i ; 1\le i\le n\hbox{ and } z_i\in [0,\gamma)\}|.$$

We recall the following standard definition:

\begin{definition}
For a finite sequence $z=(z_1,\ldots,z_n)$, we define the
{\it star discrepancy $D_n^*=D_n^*(z_1,\ldots,z_n)$} as
$$\sup_{0<\gamma\le 1}\left|{A([0,\gamma),z)\over n}-\gamma\right|.$$
Given an infinite sequence $w=(w_1,w_2,\ldots)$, we define
$$
D_n^*(w)=D_n^*(w_1,w_2,\ldots,w_n).
$$
\end{definition}

For convenience, set $D^*(z_1,\ldots,z_n)=D_n^*(z_1,\ldots,z_n)$.
Obviously, this definition does not depend on the order that
the $z_i$'s are chosen in forming $z$.  We will use this fact to reorder a sequence into an increasing sequence so that we may compute its star discrepancy with the following lemma from \cite{KuN}:

\begin{lemma}\labl{kn1}
If $0\le z_1\le\cdots\le z_n < 1$, then an upper bound for the star discrepancy
$D_n^*(z_1,\ldots,z_n)$ is given by
$${1\over 2n}+\max_{1\le i\le n}\left|z_i-{2i-1\over 2n}\right|.$$
\end{lemma}

We note that by \refl{kn1}, $\frac {1} {2n} \leq D_n^*(z) \leq 1$ for all sequences~$z=(z_1,z_2,\ldots,z_n)$
with $z_j$ in $[0,1)$ for all $j$.
It is well known that an infinite sequence
$z=(z_1,\ldots,z_n,\ldots)$ is u.d. mod $1$ iff
$\lim_{n\to\infty}D_n^*(z_1,\ldots,z_n)=0$. This fact and \refl{kn1}
will allow us to prove $Q$-distribution normality of
a well chosen $Q$ and $x$ by computing upper bounds on star discrepancies.

We recall the following lemma from \cite{KuN}:

\begin{lemma}
If~$t$ is a positive integer and for $1\le j\le t$, 
$z_j$ is a finite sequence in $[0,1)$ with
star discrepancy at most~$\epsilon_j$, then
$$D^*(z_1 z_2 \cdots z_t)\le {\sum_{j=1}^t |z_j| \epsilon_j \over\sum_{j=1}^t |z_j|}.$$
\end{lemma}

\begin{corollary}\labc{kn2}
If~$t$ is a positive integer and for $1\le j\le t$,
$z_j$ is a finite sequence in $[0,1)$ with
star discrepancy at most~$\epsilon_j$, then
$$D^*(l_1 z_1 \cdots l_t z_t)\le {\sum_{j=1}^t l_j |z_j| \epsilon_j \over\sum_{j=1}^t l_j |z_j|}.$$
\end{corollary}

We note the following simple lemma:

\begin{lemma}\labl{mono}
Let $U$ and $U'$ be subsets of $\mathbb{R}$
such that $U$ has a maximum $M$ and a minimum $m$.
If $f:U\to U'$ is a monotone function, then
$|f|$ has a maximum on $U$, which is either
$f(m)$ or $f(M)$.
\end{lemma}
\begin{proof}
Without loss of generality we
may assume that $f$ is increasing. Therefore $f$ has a minimum at $m$ and a
maximum at $M$. If $f(m) \ge 0$, then $f(x) \ge 0$ for all $x$ in $U$. This
means that $|f|=f$ is increasing on $U$. Therefore $|f|$ attains a maximum at
$M$. Similarly, if $f(M)\le 0$, then $f(x)\le 0$ for all $x$
in $U$.  This implies that $|f|=-f$ is decreasing on $U$. Therefore $|f|$ attains
a maximum at $m$.

The remaining case is that $f(m)<0<f(M)$. Let $U_A$ be the set of all $x$
in $U$ such that $f(x)\le 0$ and let $U_B$ be the set of all $x$ in $U$ such
that $f(x)\ge 0$. Note that $|f|$ is decreasing on $U_A$ and therefore
$f|U_A$ has a maximum at $m$. Similarly, $|f|$ is increasing on $U_B$ and
therefore $f|U_B$ has a maximum at $M$. Since $U=U_A\cup U_B$, it follows
that $|f|$ has a maximum at $m$ or $M$.
\end{proof}

\begin{lemma}\labl{e1l}
Let $x=(E_1,\ldots, E_n)$ be an $(\e,1,\la_b)$-normal block in base $b$.
If $y=(E_1/b,\ldots,E_n/b)$, then
$$D^*(y)\le {1\over b}+\e+\frac {1} {|x|}.$$
\end{lemma}

\begin{proof}
We wish to apply \refl{kn1} to bound $D^*(y)$.
However, \refl{kn1} only applies to
increasing sequences in $[0,1)$, so we must first reorder the
sequence~$y$.
Let $z=(z_1,\ldots,z_n)$ be the sequence of values
$E_1/b,\ldots,E_n/b$ written in increasing order.
We note that each~$z_t$ has the form $j/b$ for some~$j$ in the set $\{0,1,\ldots,b-1\}$.
Since~$z$ is an increasing sequence, we may
partition the integers from~$1$ to~$n$ into intervals
$U_0,\ldots, U_{b-1}$ such that $z_t=j/b$ for~$t$ in~$U_j$.
We let~$m_j$ and~$M_j$ be the least and greatest elements of~$U_j$,
respectively.

By \refl{kn1}, we know that $D^*(z)$ is bounded above by
$${1\over 2n}+\max_{1\le t\le n}\left|z_t-{2t-1\over 2n}\right|.$$
Fix~$j$.
Note that ${2t-1\over 2n}$ is an increasing function of
$t$ on $U_j$ and $z_t$ is a constant function of $t$ on $U_j$.
Therefore $z_t-{2t-1\over 2n}$ is a decreasing function of $t$
on $U_j$. So, for each $j$, \refl{mono} shows that the expression
$\left| z_t-{2t-1\over 2n} \right|$ is maximized for
$t=m_j$ or $t=M_j$.

By \refd{1.12}, we know that $x$ is $(\e,1,\la_b)$-normal iff
for all $j$ in~$0,1,\ldots, ~{b-1}$, we have
$$
(1-\e)\ob n \le N((j),x) \le (1+\e) \ob n.
$$

Thus,
$$
m_j=
\left( \sum_{t=0}^{j-1} N((t),x) \right)+1\ge
\left( \sum_{t=0}^{j-1} (1-\e) \ob n) \right)+1
=j(1-\e)\ob n+1:=\bar m_j
$$

and
$$
M_j=
\sum_{t=0}^j N((t,x)\le
\sum_{t=0}^j (1+\e) \ob n = (j+1)(1+\e) \ob n:=\bar M_j.
$$

Letting
$$f_j(x)=\left({j\over b}-{2x-1\over 2n}\right),$$
we see that
$$
D^*(y) \leq
{1\over 2n}+\max_{1\le t\le n}\left|z_t-{2t-1\over 2n}\right|
=\frac {1} {2n}+ \max_{0 \le j \le b-1} \max \left( \left| f_j(m_j) \right|, \left|f_j(M_j) \right| \right).
$$

Obviously, $f$ is a monotone function. Note that
$\bar m_j\le m_j\le M_j\le \bar M_j$. By \refl{mono},
the maximum of $|f_j(x)|$ on
$[\bar m_j,\bar M_j]$ occurs at $\bar m_j$ or $\bar M_j$.
Therefore
$$\max\{|f_j(m_j)|,|f_j(M_j)|\}\le \max\{|f_j(\bar m_j)|,|f_j(\bar M_j)|\}.$$
Note that
$$
|f_j(\bar m_j)|
= \left| \frac {j} {b} - \frac {2\left(j(1-\e) \ob n+1\right)-1} {2n} \right|
$$
$$
= \left| \frac {2nj-2j(1-\e)n+b} {2nb}\right|
= \left| \frac {2nj\e+b} {2nb}\right|
=\frac {j\e} {b}+ \frac {1} {2n}.
$$

Similarly, note that
$$
|f_j(\bar M_j)|
= \left| \frac {j} {b} - \frac {2(j+1)(1+\e)\ob n-1} {2n} \right|
$$
$$
=\left| \frac {2nj-2nj-2nj\e-2n-2n\e+b} {2nb} \right| 
$$
$$
\leq \left| \frac {-2nj\e-2n-2n\e} {2nb} \right| + \left| \frac {b} {2nb} \right| = \frac {j+1} {b} \e + \ob + \frac {1} {2n}.
$$

Thus $\max(|f_j(\bar m_j)|,|f_j(\bar M_j)|) \leq \frac {j+1} {b} \e + \ob + \frac {1} {2n}$ and we see that
$$
D^*(y) \leq \frac {1} {2n}+\max_{0 \leq j \leq b-1} \left( \frac {j+1} {b} \e + \ob + \frac {1} {2n} \right)
=\frac {1} {2n}+\left(\frac {b} {b} \e+ \ob+\frac {1} {2n} \right)
=\e+\ob+\frac {1} {|x|}.
$$\qed
\end{proof}

By \refl{e1l}, we know that $D^*(y_i)$ is bounded above by
$$\e_i':={1\over b_i}+\e_i+{1\over |x_i|}.$$

Given a positive integer~$n$, let $m=n-L_i$.
Note that $m$ can be written uniquely as $\al |x_{i+1}|+\be$
with $0\le\al\le l_{i+1}$ and $0\le \be< |x_{i+1}|$. We define $\alpha$
and $\beta$ as the unique integers satisfying these conditions.

Let $y=l_1 y_1 l_2 y_2\ldots $ and recall that $D^*(z)$ is bounded above by $1$ for all finite sequences $z$ of real numbers in $[0,1)$.
By \refc{kn2},
$$
D_n^*(y) \leq
f_i(\al,\be):=
\frac {l_1|x_1|\e_1'+\ldots+l_i |x_i| \e_i'+(|x_{i+1}|\e_{i+1}')\al+\be} {l_1|x_1|+\ldots+l_i |x_i|+|x_{i+1}|\al+\be}.
$$

Note that $f_i(\al,\be)$ is a rational function in  $\al$ and $\be$.
We consider the domain of $f_i$ to be
$\mathbb{R}_0^+ \times \mathbb{R}_0^+$ where $\mathbb{R}_0^+$ is the set
of all non-negative real numbers.
Now we give an upper bound for $D_n^*(y)$.
Since $D_n^*(y)$ is at most $f_i(\al,\be)$, it is enough to bound
$f_i(\al,\be)$ from above on $[0,l_{i+1}]\times [0,|x_{i+1}|]$.

\begin{lemma}\labl{boundf}
If $l_i>0$, $|x_i|>0$, $\e_{i+1}'<1$,
\begin{equation}\label{eqn:four}
l_1|x_1|+\ldots+l_{i-1}|x_{i-1}|>l_1|x_1|\e_1'+\ldots+l_{i-1}|x_{i-1}|\e_{i-1}',
\end{equation}
\begin{equation}\label{eqn:five}
\frac {|x_{i+1}|} {l_i |x_i|} < \frac {1-\e_i'} {\e_{i+1}'}
\end{equation}
and
$$
(w,z)\in \{0,\ldots,l_{i+1}\}\times\{0,\ldots,|x_{i+1}|-1\},
$$
then
$$f_i(w,z)<f_i(0,|x_{i+1}|)=\frac {l_1|x_1|\e_1'+\ldots+l_i |x_i| \e_i'+|x_{i+1}|} {l_1|x_1|+\ldots+l_i |x_i|+|x_{i+1}|}.$$
\end{lemma}

\begin{proof}
To bound~$f_i(w,z)$, we first compute its partial derivatives
$\fsz (w,z)$ and $\fsw (w,z)$.  We will show that $\fsw (w,z)$ is always negative, while $\fsz (w,z)$ is always positive. Note that this is enough to prove \refl{boundf}
 since $0\le\al$ and $\be<|x_{i+1}|$.

First, we note that $f_i (w,z)$ is a rational function of $w$ and $z$ of the form
$$
f_i (w,z)=\frac {C+Dw+Ez} {F+Gw+Hz}
$$
where
$$
C=l_1|x_1|\e_1'+\ldots+l_i|x_i|\e_i', \ D=|x_{i+1}|\e_{i+1}', \ E=1,
$$
\begin{equation}\label{eqn:ledit2}
F=l_1|x_1|+\ldots+l_i|x_i|, \ G=|x_{i+1}| \ \textrm{and \ } H=1.
\end{equation}

Therefore
$$
\frac {\partial f_i} {\partial w} (w,z)=\frac {D(F+Gw+Hz)-G(C+Dw+Ez)} {(F+Gw+Hz)^2}=\frac {D(F+Hz)-G(C+Ez)} {(F+Gw+Hz)^2};
$$
$$
\frac {\partial f_i} {\partial z} (w,z)=\frac {E(F+Gw+Hz)-H(C+Dw+Ez)} {(F+Gw+Hz)^2}=\frac {E(F+Gw)-H(C+Dw)} {(F+Gw+Hz)^2}.
$$
Thus, the sign of $\frac {\partial f_i} {\partial w} (w,z)$ does not depend on $w$ and the sign of $\frac {\partial f_i} {\partial z} (w,z)$ does not depend on $z$.
We will show that $f_i(w,z)$ is a
decreasing function of $w$ by proving that
\begin{equation}\label{eqn:one}
D(F+Hz)<G(C+Ez).
\end{equation}
Similarly, we show that $f_i (w,z)$ is an increasing function of $z$ by verifying that
\begin{equation}\label{eqn:two}
E(F+Gw)>H(C+Dw).
\end{equation}
Substituting the values in \refeq{ledit2} into \refeq{one}, we see that
$$
|x_{i+1}|\e_{i+1}'(l_1|x_1|+\ldots+l_i|x_i|+z)<|x_{i+1}|(l_1|x_1|\e_1'+\ldots+l_i|x_i|\e_i'+z).
$$
Since $|x_{i+1}| \ge |x_i| >0$, we may divide both sides by $|x_{i+1}|$ to obtain
\begin{equation}\label{eqn:three}
l_1|x_1|\e_{i+1}'+\ldots+l_i|x_i|\e_{i+1}'+z \epsilon_{i+1}' < l_1|x_1|\e_1'+\ldots+l_i|x_i|\e_i'+z.
\end{equation}
So, we only have to show \refeq{three}, which is true since
$$\e_i'={1\over b_i}+\e_i+\frac {1} {|x_i|}$$
is decreasing and $\e_{i+1}'<1$.

Also, by substituting the values in \refeq{ledit2} into \refeq{two}, we see that
$$
(l_1|x_1|+\ldots+l_{i-1}|x_{i-1}|)+(l_i|x_i|+w|x_{i+1}|)
>(l_1|x_1|\e_1'+\ldots+l_{i-1}|x_{i-1}|\e_{i-1}')+(l_i|x_i|\e_i'+|x_{i+1}|\e_{i+1}').
$$
By condition \refeq{four} we know that
$$
l_1|x_1|+\ldots+l_{i-1}|x_{i-1}|>l_1|x_1|\e_1'+\ldots+l_{i-1}|x_{i-1}|\e_{i-1}'.
$$
Therefore it is enough to show
$$
l_i|x_i|+w|x_{i+1}|>l_i|x_i|\e_i'+|x_{i+1}|\e_{i+1}'.
$$
Since $l_i |x_i|$ is the smallest possible value of $l_i|x_i|+w|x_{i+1}|$ for non-negative $w$, we need only show that
$$
l_i|x_i|>l_i|x_i|\e_i'+|x_{i+1}|\e_{i+1}'.
$$
By routine algebra, this is equivalent to
$$
\frac {|x_{i+1}|} {l_i |x_i|} < \frac {1-\e_i'} {\e_{i+1}'},
$$
which is true by \refeq{five}.\qed
\end{proof}

Set
$$\bar\e_i=
f_i(0,|x_{i+1}|)
=
\frac {l_1|x_1|\e_1'+\ldots+l_i |x_i| \e_i'+|x_{i+1}|} {l_1|x_1|+\ldots+l_i |x_i|+|x_{i+1}|}.
$$

\begin{lemma}\labl{e0}
$\lim_{n\to\infty}\bar\e_{i(n)}=0$.
\end{lemma}

\begin{proof}
We write~$i$ for $i(n)$ throughout. For $i$ large enough, we have
$$
\frac {l_1|x_1|\e_1'+\ldots+l_i |x_i| \e_i'+|x_{i+1}|} {l_1|x_1|+\ldots+l_i |x_i|+|x_{i+1}|}<\frac {l_1|x_1|\e_1'+\ldots+l_i |x_i| \e_i'+|x_{i+1}|} {l_i |x_i|}
$$
$$
=\frac {l_1|x_1|\e_1'+\ldots+l_{i-1} |x_{i-1}| \e_{i-1}'} {l_i |x_i|}+\e_i'+\frac {|x_{i+1}|} {l_i |x_i|}<\frac {l_{i-1}|x_{i-1}|} {l_i |x_i|}\cdot i \cdot \e_{i-1}'+\e_i'+\frac {|x_{i+1}|} {l_i |x_i|},
$$
where the last inequality uses the fact that $\e_i'$ is decreasing.
Note that
$\frac {l_{i-1}|x_{i-1}|} {l_i |x_i|}\cdot i \cdot \e_{i-1}'\to 0$ by \refeq{nice1},
$\e_i'={1\over b_i}+\e_i + \frac {1} {|x_i|} \to 0$
and $\frac {|x_{i+1}|} {l_i |x_i|}\to 0$ by \refeq{nice2}.
Therefore $\lim_{i \rightarrow \infty} \bar \e_i=0$.  Since $i$ can be made arbitrarily large by choosing large enough $n$, the lemma follows.\qed
\end{proof}

\begin{theorem}\labt{mqd}
If~$V$ is  an $MFF$ and $\{x_i\}_{i=1}^\infty$ is
a $V$-nice sequence, then~$x$ is $Q$-distribution normal.
\end{theorem}

\begin{proof}
By \reft{Salat}, it is enough to show that $D_n^*(y) \rightarrow 0$.
Since $x_i$ is $(\e_i,1,\la_{b_i})$-normal, we see that
$D^*(y_i)\le\e_i'$ by \refl{e1l}.
We wish to apply \refc{kn2} and for large enough~$i$ apply \refl{boundf}
as well. To apply \refl{boundf} for large $i$, we need
only prove several inequalities for large $i$.
In applying these inequalities, we will have $i=i(n)$ as defined
in \refeq{idef}, so it is worth noting that~$i$ may be chosen as large
as one likes by choosing a large enough~$n$.

For the first inequality, note that
$\lim_{i\to\infty}l_i|x_i|=\infty$. For large
enough~$i$, the product $l_i|x_i|$ is nonzero.
For the second, we have $|x_i|>0$.
For the third inequality, $\e_{i+1}'<1$ for large enough $i$ as $\e_i' \rightarrow 0$.
Next, since $l_{i-1}|x_{i-1}|$ asymptotically dominates
$l_{i-1}|x_{i-1}|\e_{i-1}'$,
it follows that
$l_1|x_1|+\ldots+l_{i-1}|x_{i-1}|$
asymptotically dominates
$l_1|x_1|\e_1'+\ldots+l_{i-1}|x_{i-1}|\e_{i-1}'$ as well.
In particular, for large enough~$i$, we have
\begin{equation}
l_1|x_1|+\ldots+l_{i-1}|x_{i-1}|>l_1|x_1|\e_1'+\ldots+l_{i-1}|x_{i-1}|\e_{i-1}'.
\end{equation}

Finally, for the fifth inequality, noting that
$$\lim_{i\to\infty}{|x_{i+1}|\over l_i|x_i|}=0$$ and
that
$$\lim_{i\to\infty}{1-\e_i'\over \e_{i+1}'}=\infty$$
since $\lim_{i\to\infty}\e_i'=0$, we see that
$$
{|x_{i+1}|\over l_i|x_i|}<
{1-\e_i'\over \e_{i+1}'}
$$
for large~$i$.

So, for large enough~$i$, $D_n^*(y)\le \bar\e_i$ and
$\lim_{i\to\infty}\bar\e_i=0$. Thus $\lim_{n\to\infty}D_n^*(y)=0$.  \qed
\end{proof}

Let $C_{b,w}$ be the block formed by concatenating all the blocks of length $w$ in base $b$ in lexicographic order.  For example,
$$
C_{3,2}=1(0,0)1(0,1)1(0,2)1(1,0)1(1,1)1(1,2)1(2,0)1(2,1)1(2,2)
$$
$$
=(0,0,0,1,0,2,1,0,1,1,1,2,2,0,2,1,2,2).
$$
Let $x_1=(0,1)$, $b_1=2$ and $l_1=0$.  For $i \geq 2$, let $x_i=C_{i,i^2}$, $b_i=i$ and $l_i=i^{3i}$.
It was shown in \cite{Mance} that $x$ is $Q$-normal. We now show that $x$ is $Q$-distribution
normal.

\begin{theorem}\labt{qde}
Let $x_1=(0,1)$, $b_1=2$ and $l_1=0$.
If for $i\ge 2$, we let $x_i=C_{i,i^2}$, $b_i=i$,
and $l_i=i^{3i}$, 
then~$x$ is $Q$-distribution normal.
\end{theorem}

\begin{proof}
We let $\e_1=3/5$. For $i\ge 2$, we let $\e_i=1/i$.
By \cite{Mance}, we know that $x_i$ is $(\e_i,1,\la_{b_i})$-normal.
It is enough to show \refeq{nice1} and \refeq{nice2}.
Note that trivially, \refeq{good2} implies \refeq{nice1} and \refeq{good3} implies
\refeq{nice2}. Moreover, it was proven in \cite{Mance} that \refeq{good2} and \refeq{good3}
hold.
Therefore $x_i$ is $V$-nice. So, by \reft{mqd}, we see that~$x$ is
$Q$-distribution normal as claimed.\qed
\end{proof}

\begin{acknowledgements}
We thank the referees for their many valuable suggestions.  We would also like to thank  Vitaly Bergelson for pointing us in the direction of this problem and Marc Carnovale and Laura Harvey for their help in editing this paper.
\end{acknowledgements}

\end{document}